\documentclass[aip,graphicx]{revtex4-1}

\draft % marks overfull lines with a black rule on the right
\usepackage{microtype,booktabs}
\usepackage{amsmath,amssymb,amsthm}
\usepackage{mathrsfs}
\numberwithin{equation}{section}

\begin{document}

%\preprint{}

\title{Formulation of Stochastic Contact Hamiltonian Systems } %Title of paper



\author{Pingyuan Wei}
\email{weipingyuan@hust.edu.cn}
%\thanks{}
\affiliation{School of Mathematics and Statistics \& Center for Mathematical Sciences, Huazhong University of Science and Technology, Wuhan 430074,  China}

\author{Zibo Wang}
\email{(Corresponding author) zibowang@hust.edu.cn}
%\thanks{}
\affiliation{School of Mathematics and Statistics \& Center for Mathematical Sciences, Huazhong University of Science and Technology, Wuhan 430074,  China}


\date{\today}

\newtheorem{thm}{Theorem}[section]
\newtheorem{cor}[thm]{Corollary}
\newtheorem{lem}[thm]{Lemma}
\newtheorem{lemma}[thm]{Lemma}
\theoremstyle{remark}
\newtheorem{rem}[thm]{Remark}
\newtheorem{remark}[thm]{Remark}
\theoremstyle{definition}
\newtheorem{defn}[thm]{Definition}
\newtheorem{example}[thm]{Example}
\newcommand{\ZZ}{\mathbb{Z}}
\newcommand{\QQ}{\mathbb{Q}}
\newcommand{\NN}{\mathbb{N}}
\newcommand{\RR}{\mathbb{R}}
\newcommand{\CC}{\mathbb{C}}
\newcommand{\TT}{\mathbb{T}}
\newcommand{\EE}{\mathbb{E}}
\newcommand{\PP}{\mathbb{P}}
\newcommand{\op}{\operatorname}
\providecommand{\abs}[1]{\lvert#1\rvert}
\providecommand{\norm}[1]{\lVert#1\rVert}

\begin{abstract}
In this work we devise a stochastic version of contact Hamiltonian systems, and show that the phase flows of these systems preserve contact structures. Moreover, we provide a sufficient condition under which these stochastic contact Hamiltonian systems are completely integrable. This establishes an appropriate framework for investigating stochastic contact Hamiltonian systems.\\


{\bf It is well known that contact Hamiltonian systems are the ``odd-dimensional cousins" of symplectic Hamiltonian systems. Contact Hamiltonian systems arise in the description of dissipative systems, and have found applications in many areas of science. Motivated by Bismut's work about stochastic Hamiltonian systems on symplectic manifolds that preserve symplectic structures, we propose a stochastic version of contact Hamiltonian systems which preserve contact structures. We also analyse when a stochastic contact Hamiltonian system is completely integrable.  Two examples are presented to illustrate our results.
%As a cousin of symplectic Hamiltonian system, contact Hamiltonian system is famous in the study of thermodynamics. Recently, it has been found useful in several areas ranging from classical and statistical mechanics to quantum mechanics. In order to consider the random influence of the external environment on the system, people introduce noise into the system. In this paper, we study a special kind of stochastic contact Hamiltonian system which preserves contact structure. That is the phase flow preserves the contact form up to multiplication by a conformal factor. Finally, we introduce a completely integrable condition under this setting. Generalized action angle coordinates also introduced in example.
}

%We devise a a stochastic version of contact Hamiltonian system which preserves contact structure (up to multiplication by a conformal factor). Moreover, we provide  a sufficient condition under which these stochatsic contact HS are completely integrable.
%and show that the phase flows of these systems consist of conformal contactomorphisms up to multiplication by a conformal factor, that is, $\phi_t^\ast\eta=\lambda_t\eta$ for some smooth family of functions $\lambda_t:M\to \RR$.
%possess the property of preserving contact structure (up to multiplication by a conformal factor).

\end{abstract}


\maketitle %\maketitle must follow title, authors, abstract and \pacs

\section{Introduction}

Hamiltonian systems on symplectic manifolds are the natural framework for classical mechanics and statistical physics, and the starting point for their quantum counterparts \cite{Bravetti2017entropy}. As models for conservative systems, they have been extensively studied \cite{Arnold2013mathematical}. Hamiltonian systems on contact manifolds are an appropriate scenario for certain dissipative systems \cite{DeLeon2019}. Contact Hamiltonian systems are not only an odd dimensional counterpart of symplectic cases, but also contain rich geometric structures. In recent years, contact Hamiltonian systems have attracted a lot of attentions. For example, they are used for the study of thermodynamics \cite{Bravetti2019contact}, optimal control \cite{Ramirez2016partial} and fluid mechanics \cite{Ghrist2007contact}. Some theories such as weak KAM theory \cite{Wang2019aubry}, complete integrability \cite{Boyer2011completely,Visinescu2017contact} and variational principle \cite{Wang2019variational,Liu2018contact} for contact Hamiltonian systems have been developed recently. Contact algorithms are also proposed to simulate such systems \cite{Vermeeren2019contact}.

Stochastic dynamical systems can be regarded as models that take into account the influence of the random environment \cite{Duan}. There are recent studies on Hamiltonian systems with random perturbations; for example, Brin-Freidlin \cite{Brin2000stochastic} and MacKay \cite{mackay2010langevin}. A stochastic Hamiltonian system preserving symplectic structure was proposed by Bismut \cite{Bi}. Following this framework, Li \cite{Li} and Wei \cite{Pingyuan2019} discussed averaging principles for completely integrable stochastic Hamiltonian systems. Moreover, the symplectic structure-preserving numerical algorithms are developed. For example, Wang et al. \cite{wang2009dynamics} designed a stochastic variational integrator using Hamilton's principle.

Integrability is a topic worth discussing in dynamical systems. The famous Arnold-Liouville theorem describes a connection between completely integrable Hamiltonian systems and toric geometry in the symplectic setting\cite{Boyer2011completely}. Under the hypothesis of Arnold-Liouville theorem, a Hamiltonian system can be expressed by the action-angle coordinate, such that a first integral only depends on the action variable, while the angle variable changes on torus\cite{Arnold2013mathematical}. As mentioned earlier, a contact version of complete integrability has been stuided recently.

In this paper, we devise a class of stochastic contact Hamiltonian systems which preserve contact structures. Then we provide a sufficient condition under which these stochastic contact Hamiltonian systems are completely integrable.

This paper is organized as follows. In section 2, we present a class of stochastic contact Hamiltonian systems. In section 3, we demonstrate that the phase flow of these systems preserve contact structures. The key idea of the proof is a chain rule that holds for Stratonovitch differentials. In section 4, we show that a class of stochastic contact Hamiltonian systems are completely integrable. Finally, we present two illustrative examples. In particular, we establish the generalized contact action-angle coordinates in the second example.

%%---------------------------------------------------------------------------------------------------------------
\section{Stochastic Contact Hamiltonian Systems}

A smooth $(2n+1)$-dimensional differential manifold $M$ is said to be a contact manifold if it is equipped with a contact structure, which is a nondegenerate $1$-form $\eta$ such that $\eta\wedge (d\eta)^n$ is a volume form (i.e., it is nonzero at each point of $M$).

Given a smooth (Hamiltonian) function $H_0$ and a family of $d$ smooth (Hamiltonian) functions $\{ H_k \}_{k=1}^{d}$ on $M$. The corresponding contact Hamiltonian vector fields, denoted by $X_{H_0}$ and $X_{H_k}$ $(k=1,2,...,d)$, are defined through the following intrinsic relations \cite{Bravetti2017contact}
\begin{align}\label{Ham}
H_i=\iota_{X_{H_i}}\eta\;\;\;\;\text{and}\;\;\;\;dH_i=-\iota_{X_{H_i}}d\eta+\mathcal{R}(H_i)\eta
\end{align}
for $i=0,1,\cdots,d$, where $\mathcal{R}$ is Reeb vector field (which is the unique vector field such that $\iota_{\mathcal{R}}\eta=1$ and $\iota_{\mathcal{R}}d\eta=0$).

We propose the following stochastic contact Hamiltonian system:
\begin{equation}\label{Equation-1}
dx_t=X_{H_0}(x_t)dt+\sum_{k=1}^d X_{H_k} (x_t) \circ dB_t^k , ~ x(t_0)=x_0\in M,
\end{equation}
where $\{B_t^k\}_{k=1}^d$ are pairwise independent Brownian motions on a probability space $( \Omega , \mathscr{F}, P)$ and $``\circ"$ stands for Stratonovitch differentials. We call $X_{H_0}$ the drift vector field and $X_{H_k}$ the diffusion vector fields.

\begin{remark}
Notice that we have chosen to write \eqref{Equation-1} in Stratonovich rather than It\^o stochastic differentials. This is mainly because Stratonovich differentials has the advantage of leading to ordinary chain rule of the Newton-Leibniz type under coordinates transformation. Such a property offers some reduction in calculations, and makes the Stratonovich differential natural to use especially in connection with stochastic differential equations on manifolds.
We also note that Brownian motions in \eqref{Equation-1} can be extended to more general stochastic processes, for example L\'evy processes or semimartingales. Accordingly, the Stratonovich differentials should be replaced by Marcus differentials.

\end{remark}

\begin{remark}
Formally, equation \eqref{Equation-1} can be rewritten in the form of a Hamiltonian system
$
\frac{d}{dt}x_t=X_{\tilde{H}}(x_t)
$
with a ``randomized'' contact Hamiltonian $\tilde{H}=H_0+\sum_{k=1}^dH_kw_t^k$, where $w_t$ is an $d$-dimensional noise (which is a Gaussian white noise for the Brownian case). One may regard this as a contact Hamiltonian within the external world. A ``deterministic" contact Hamiltonian system with Hamiltonian $H_0$ is considered as a dissipative mechanical system. The stochastic part in \eqref{Equation-1} is introduced to characterize the complicated interaction between the deterministic system and the fluctuating environment.
Nevertheless, in the next section, we will show that the stochastic flow for equation (2.2) still preserves the contact strucure. If the stochatsic part in (2.2) takes other forms, the corrsponding stochastic flow may not preserve the contcat structure.
\end{remark}

In view of some concrete applications, it is convenient to write system \eqref{Equation-1} in local coordinates. According to Darboux theorem\cite{Geiges2008} for contact manifolds, around each point in $M$, one can find local (called Darboux  or canonical) coordinates $(q,p,z)=(q_1,\cdots,q_n,p_1,\cdots,p_n,z)$ such that
\begin{equation}\label{Darboux}
\eta=dz-pdq,\;\;\;\;\mathcal{R}=\frac{\partial}{\partial z}.
\end{equation}
Here $pdq$ has to be read as $\sum_{j=1}^np_jdq_j$. This same shorthand notation applies to analogous expressions below.

In Darboux coordinates, we get the following local expressions for Hamilton vector fields
%\begin{align}\label{Darboux vector field}
%X_{H_i}=\sum_{j=1}^n\frac{\partial H_i}{\partial p_j}\frac{\partial}{\partial q_j}-\left( \frac{\partial H_i}{\partial q_j}+p_j\frac{\partial H_i}{\partial z}\right)\frac{\partial}{\partial p_j}+\left(p_j\frac{\partial H_i}{\partial p_j}-H_i\right)\frac{\partial}{\partial z},
%\end{align}
\begin{align}\label{Darboux vector field}
X_{H_i}=\frac{\partial H_i}{\partial p}\frac{\partial}{\partial q}-\left( \frac{\partial H_i}{\partial q}+p\frac{\partial H_i}{\partial z}\right)\frac{\partial}{\partial p}+\left(p\frac{\partial H_i}{\partial p}-H_i\right)\frac{\partial}{\partial z},
\end{align}
for $i=0,1,\cdots,d$. Therefore, a canonical stochastic contact Hamiltonian system can be written as
\begin{align}
dq&=\frac{\partial H_0}{\partial p}dt+\sum_{k=1}^d \frac{\partial H_k}{\partial p}  \circ dB_t^k,  \label{canonical eq1}\\
dp&=-\left( \frac{\partial H_0}{\partial q}+p\frac{\partial H_0}{\partial z}\right)dt-\sum_{k=1}^d \left( \frac{\partial H_k}{\partial q}+p\frac{\partial H_k}{\partial z}\right)  \circ dB_t^k, \label{canonical eq2}\\
dz&=\left(p\frac{\partial H_0}{\partial p}-H_0\right)dt +\sum_{k=1}^d \left(p\frac{\partial H_k}{\partial p}-H_k\right)\circ dB_t^k \label{canonical eq3}
\end{align}
with initial state $(q(t_0),p(t_0),z(t_0))\triangleq(\alpha,\beta,\gamma)$, $t_0\geqslant 0$. %For convenience, we call $q$ the position, $p$ the momentum and $z$ a scalar quantity that

\section{Conformal Contactomorphism}
As in the deterministic case \cite{Geiges2008}, we are interested in the system \eqref{canonical eq1}-\eqref{canonical eq3} such that the solution mapping $(\alpha, \beta,\gamma) \to (p,q,z)$ preserves the natural contact structure (leaves the contact form invariant) up to multiplication by a conformal factor. That is, the phase flow $\phi_t$ of the system is a conformal contactomorphism %(also called, contact isotopy)
in the sense
\begin{equation}
\phi_t^\ast\eta=\lambda_t\eta
\end{equation}
 for a nowhere zero function (i.e., conformal factor) $\lambda_t:M\to \RR$. In local coordinates, we only need to focus on
 \begin{equation}\label{preserve}
 dz-pdq=\lambda_t(d\gamma-\beta d \alpha).
 \end{equation}
Note that the conformal factor $\lambda_t$, with $t\in[t_0,\infty)$, is a smooth family with $\lambda_{t_0}=1$. Hence all $\lambda_t$ take value in $\RR^+$. To avoid confusion, we should note that the differentials in \eqref{canonical eq1}-\eqref{canonical eq3} and \eqref{preserve} have different meanings: In \eqref{canonical eq1}-\eqref{canonical eq3}, $p,q,z$ are treated as functions of time, while, in \eqref{preserve}, the differentiation is made with respect to the initial data $\alpha,\beta,\gamma$.

\begin{remark}
For deterministic case, i.e., $\dot{x}=X_{H_0}(x)$, the corresponding phase flow $\psi_t$ is a conformal contactomorphism. Indeed, by Cartan's identity and \eqref{Ham}, we have $\mathcal{L}_{X_{H_0}}\eta=-\mathcal{R}(H)\eta$. Note that $\frac{\partial}{\partial t}\psi_t^\ast\eta=\psi_t^\ast\mathcal{L}_{X_{H_0}}\eta=\psi_t^\ast(-\mathcal{R}(H)\eta)=(-\mathcal{R}(H)\circ \psi_t)\psi_t^\ast\eta$. We conclude that $\psi_t^\ast\eta=\lambda_t\eta$ with $\lambda_t=\exp \big(- \int_{t_0}^t ( \mathcal{R}(H)\circ \psi_\tau )\big)d\tau$. We also remark that a map $f$ is called a (strict) contactomorphism if $f^\ast\eta\equiv\eta$.
\end{remark}

\begin{thm} {\bf (Preserving the Contact Structure)}
The phase flow of the stochastic contact Hamiltonian system \eqref{canonical eq1}-\eqref{canonical eq3} is a conformal contactomorphism (i.e., it preserves the contact structure) with a conformal factor
 \begin{equation}\label{factor}
\lambda_t=\exp{\big(-\int_{t_0}^t\frac{\partial}{\partial z}H_0d\tau-\sum_{k=1}^d\int_{t_0}^t\frac{\partial}{\partial z}H_k\circ dB_\tau^k\big)}.
 \end{equation}
%$$\lambda_t=\exp\Big[-\int_0^t\frac{\partial}{\partial z}H_0(q_\tau,p_\tau,z_\tau)d\tau-\sum_{k=1}^d\int_0^tH_k(q_\tau,p_\tau,z_\tau)\circ dB_t^\tau\Big].$$
\end{thm}

\begin{proof}
Notice that the condition \eqref{preserve} for phase flow to be conformal contactomorphism can be rewritten as
\begin{align}\label{preserve1}
\left\{
\begin{array}{rl}
\frac{\partial z}{\partial \alpha_{r}}-\sum_{j=1}^np_j\frac{\partial q_j}{\partial \alpha_{r}}=&-\lambda_t \beta_r,  \;\;\;\;  r=1,\cdots,n, \\
\frac{\partial z}{\partial \beta_{r}}-\sum_{j=1}^np_j\frac{\partial q_j}{\partial \beta_{r}}=&0,  \;\;\;\;\;\;\;\;\;\;\;  r=1,\cdots,n, \\
\frac{\partial z}{\partial \gamma}-\sum_{j=1}^np_j\frac{\partial q_j}{\partial \gamma}=&\lambda_t . \\
\end{array}
\right.
\end{align} %%------------------------------
%\textcolor{red}{
%\begin{align}%\label{eq1}
%i.e.,\;\left\{
%\begin{array}{rl}
%z_\alpha^r-\sum_{j}p_jq_\alpha^{jr}=&-\lambda_t \beta_r,  \;\;\;\;  r=1,\cdots,n, \\
%z_\beta^r-\sum_{j}p_jq_\beta^{jr}=&0,  \;\;\;\;\;\;\;\;\;\;\;  r=1,\cdots,n, \\
%z_\gamma-\sum_{j}p_jq_\gamma^j=&\lambda_t . \\
%\end{array}
%\right.
%\end{align}
%\begin{align}\label{preserve2}
%\left\{
%\begin{array}{rl}
%\frac{d}{dt}\frac{\partial z}{\partial \alpha_{r}}-\left(\dot{p}\frac{\partial q}{\partial \alpha_{r}}+p\frac{d}{dt}\frac{\partial q}{\partial \alpha_{r}}\right)=&-\left(\dot{\lambda}_t \beta_r +\lambda_t \dot{\beta}_r\right),  \;\;\;\;  r=1,\cdots,n, \\
%\frac{d}{dt}\frac{\partial z}{\partial \beta_{r}}-\left(\dot{p}\frac{\partial q}{\partial \beta_{r}}+p\frac{d}{dt}\frac{\partial q}{\partial \beta_{r}}\right)=&0,  \;\;\;\;\;\;\;\;\;\;\;  r=1,\cdots,n, \\
%\frac{d}{dt}\frac{\partial z}{\partial \gamma}-\left(\dot{p}\frac{\partial q}{\partial \gamma}+p\frac{d}{dt}\frac{\partial q}{\partial \gamma}\right)=&\dot{\lambda}_t.\\
%\end{array}
%\right.
%\end{align}
%}
For convenience, we adopt the following notation (for $j,r=1,\cdots,n$):
\begin{align}\label{notation-Hpqz2}
q_\alpha^{jr}=\frac{\partial q_j}{\partial \alpha_r},  \;\;q_\beta^{jr}=\frac{\partial q_j}{\partial \beta_r}, \; \;   q_\gamma^{j}=\frac{\partial q_{j}}{\partial \gamma} .%\;\text{ with }\; a=q,p,z.
\end{align}
Similarly, we define $p_\alpha^{jr},p_\beta^{jr},p_\gamma^{j},z_\alpha^{r},z_\beta^r$ and $z_\gamma$.

By calculating at $(q,p,z)=(q(t;t_0,\alpha,\beta,\gamma),p(t;t_0,\alpha,\beta,\gamma),z(t;t_0,\alpha,\beta,\gamma))$ which is a solution to system \eqref{canonical eq1}-\eqref{canonical eq3}, we conclude that  $p_\gamma^{j}, q_\gamma^{j}$, $z_\gamma$ ($j=1,\cdots,n$) satisfy the following system of stochastic differential equations
%%-------------------------
\begin{align}
dq_\gamma^{j}=&\sum_{l=1}^n(\frac{\partial^2 H_0}{\partial p_j \partial q_l}q_\gamma^{l}+\frac{\partial^2 H_0}{\partial p_j \partial p_l}p_\gamma^{l}+\frac{\partial^2 H_0}{\partial p_j \partial z}z_\gamma)dt \notag\\
&+ \sum_{k=1}^d\sum_{l=1}^n(\frac{\partial^2 H_k}{\partial p_j \partial q_l}q_\gamma^{l}+\frac{\partial^2 H_k}{\partial p_j \partial p_l}p_\gamma^{l}+\frac{\partial^2 H_k}{\partial p_j \partial z}z_\gamma)\circ dB_t^k,
\label{SDE11} \\
dp_\gamma^{j}=&-\sum_{l=1}^n\Big[\Big(\frac{\partial^2 H_0}{\partial q_j \partial q_l}+p_j\frac{\partial^2 H_0}{\partial z\partial q_l}\Big)q_\gamma^{l}+\Big(\frac{\partial^2 H_0}{\partial q_j \partial p_l}+\delta_{jl}\frac{\partial H_0}{\partial z}+p_j\frac{\partial^2 H_0}{\partial z \partial p_l}\Big)p_\gamma^{l}\notag\\
&\;\;\;\;+\Big(\frac{\partial^2 H_0}{\partial q_j \partial z}+p_j\frac{\partial^2 H_0}{\partial z^2}\Big)z_\gamma\Big]dt \notag\\
&- \sum_{k=1}^d\sum_{l=1}^n\Big[\Big(\frac{\partial^2 H_k}{\partial q_j \partial q_l}+p_j\frac{\partial^2 H_k}{\partial z\partial q_l}\Big)q_\gamma^{l}+\Big(\frac{\partial^2 H_k}{\partial q_j \partial p_l}+\delta_{jl}\frac{\partial H_k}{\partial z}+p_j\frac{\partial^2 H_k}{\partial z \partial p_l}\Big)p_\gamma^{l}\notag\\
&\;\;\;\;+\Big(\frac{\partial^2 H_k}{\partial q_j \partial z}+p_j\frac{\partial^2 H_k}{\partial z^2}\Big)z_\gamma\Big]\circ dB_t^k,
\label{SDE12}\\
dz_\gamma=&\sum_{l=1}^n\Big[\Big(\sum_{j=1}^n p_j\frac{\partial^2 H_0}{\partial p_j\partial q_l}-\frac{\partial H_0}{\partial q_l}\Big)q_\gamma^{l}+\Big(\sum_{j=1}^n p_j\frac{\partial^2 H_0}{\partial p_j\partial p_l}\Big)p_\gamma^{l}\notag\\
&\;\;\;\;+\Big(\sum_{j=1}^n p_j\frac{\partial^2 H_0}{\partial p_j\partial z}-\frac{\partial H_0}{\partial z}\Big)z_\gamma\Big]dt \notag\\
&+ \sum_{k=1}^d\sum_{l=1}^n\Big[\Big(\sum_{j=1}^n p_j\frac{\partial^2 H_k}{\partial p_j\partial q_l}-\frac{\partial H_k}{\partial q_l}\Big)q_\gamma^{l}+\Big(\sum_{j=1}^n p_j\frac{\partial^2 H_k}{\partial p_j\partial p_l}\Big)p_\gamma^{l}\notag\\
&\;\;\;\;+\Big(\sum_{j=1}^n p_j\frac{\partial^2 H_k}{\partial p_j\partial z}-\frac{\partial H_k}{\partial z}\Big)z_\gamma\Big]\circ dB_t^k,\label{SDE13}
\end{align}
with initial state $((q_\alpha^{jr}),(p_\alpha^{jr}),z_\alpha^{r})=(\delta_{jr},0,0)$.

Keeping in mind the initial state $(q(t_0),p(t_0),z(t_0))=(\alpha,\beta,\gamma)$, it is clear that $\lambda_{t_0}=1$. The third condition in \eqref{preserve1} is fulfilled if and only if
\begin{align}\label{preserve23}
dz_\gamma(t)-\sum_j dp_j(t)\cdot q^j_\gamma(t)-\sum_j p_j\cdot dq_\gamma(t)=d\lambda_t.
\end{align}
Due to \eqref{canonical eq2}, \eqref{SDE11} and \eqref{SDE13}, we conclude that the relation \eqref{preserve23} becomes
% the coefficients of left hand side of \eqref{preserve23} are equal to (for $i=0,1,\cdots,d$)
%\begin{align}
%&\sum_{l=1}^n\Big[\Big(\sum_{j=1}^n p_j\frac{\partial^2 H_i}{\partial p_j\partial q_l}-\frac{\partial H_i}{\partial q_l}\Big)q_\gamma^{l}+\Big(\sum_{j=1}^n p_j\frac{\partial^2 H_i}{\partial p_j\partial p_l}\Big)p_\gamma^{l}\notag\\
%&\;\;\;\;+\Big(\sum_{j=1}^n p_j\frac{\partial^2 H_i}{\partial p_j\partial z}-\frac{\partial H_i}{\partial z}\Big)z_\gamma\Big] \notag\\
%+&\sum_{j=1}^n\left( \frac{\partial H_i}{\partial q_j}+p_j\frac{\partial H_i}{\partial z}\right)q^j_\gamma\notag\\
%-&\sum_{j=1}^n p_j\sum_{l=1}^n(\frac{\partial^2 H_i}{\partial p_j \partial q_l}q_\gamma^{l}+\frac{\partial^2 H_i}{\partial p_j \partial p_l}p_\gamma^{l}+\frac{\partial^2 H_i}{\partial p_j \partial z}z_\gamma).
%\end{align}
%Therefore,  we conclude that
\begin{align}\label{preserve23-1}
\Big(z_\gamma -\sum_j p_jq_\gamma^{r}\Big)\left(\frac{\partial H_0}{\partial z}dt+\sum_{k=1}^d\frac{\partial H_k}{\partial z}\circ d{B}_t^k \right)&=-d{\lambda}_t.
%\frac{\partial H_0}{\partial z}(z_\gamma -\sum_j p_jq_\gamma^j)dt+\sum_{k=1}^d\frac{\partial H_k}{\partial z}(z_\gamma -\sum_j p_jq_\gamma^j)\circ dB_t^k=-d\lambda_t.
\end{align}
Similarly, the first and the second conditions in \eqref{preserve1} are fulfilled if and only if
\begin{align}%\label{preserve23-1}
%\frac{\partial H_0}{\partial z}(z_\alpha^r -\sum_j p_jq_\alpha^{jr})dt+\sum_{k=1}^d\frac{\partial H_k}{\partial z}(z_\alpha^r -\sum_j p_jq_\alpha^{jr})\circ dB_t^k=\beta_rd{\lambda}_t, \\
\Big(z_\alpha^r -\sum_j p_jq_\alpha^{jr}\Big)\left(\frac{\partial H_0}{\partial z}dt+\sum_{k=1}^d\frac{\partial H_k}{\partial z}\circ d{B}_t^k \right)&=\beta_rd{\lambda}_t, \\
%\frac{\partial H_0}{\partial z}(z_\beta^r -\sum_j p_jq_\beta^{jr})dt+\sum_{k=1}^d\frac{\partial H_k}{\partial z}(z_\beta^r -\sum_j p_jq_\beta^{jr})\circ dB_t^k=\beta_rd{\lambda}_t, \\
\Big(z_\beta^r -\sum_j p_jq_\beta^{jr}\Big)\left(\frac{\partial H_0}{\partial z}dt+\sum_{k=1}^d\frac{\partial H_k}{\partial z}\circ d{B}_t^k \right)&=0.
\end{align}
As a consequence, the condition \eqref{preserve1} holds if and only if
\begin{align}%\label{preserve23-1}
%-\lambda_t\beta_r\left(\frac{\partial H_0}{\partial z}+\sum_{k=1}^d\frac{\partial H_k}{\partial z}\circ \dot{B}_t^k \right)&=\beta_r\dot{\lambda}_t, \label{preserve-new1}\\
%0\cdot\left(\frac{\partial H_0}{\partial z}+\sum_{k=1}^d\frac{\partial H_k}{\partial z}\circ \dot{B}_t^k \right)&=0,\label{preserve-new2} \\
\lambda_t\left(\frac{\partial H_0}{\partial z}+\sum_{k=1}^d\frac{\partial H_k}{\partial z}\circ \dot{B}_t^k \right)&=-\dot{\lambda}_t, \label{preserve-new3}
\end{align}
 with $\lambda_{t_0}=1$. By solving equation \eqref{preserve-new3}, we have
 $$\lambda_t=\exp{\big(-\int_0^t\frac{\partial}{\partial z}H_0(q_\tau,p_\tau,z_\tau)d\tau-\sum_{k=1}^d\int_0^t\frac{\partial}{\partial z}H_k(q_\tau,p_\tau,z_\tau)\circ dB_\tau^k\big)}.$$
This implies that \eqref{preserve1} holds and hence conformal contactomorphism condition \eqref{preserve} is fulfilled. The proof is complete.
\end{proof}

\section{Complete Integrability}
In Bismut's \cite{Bi} setting, a class of stochastic completely integrable (symplectic) Hamiltonian systems has been studied. See more information in Li \cite{Li}. The integrability for contact Hamiltonian systems has also drawn much recent attention \cite{Arnold2013mathematical,Boyer2011completely}.

Note that although any smooth function can be chosen as a contact Hamiltonian, it is often convenient to choose the function $1=\iota_{\mathcal{R}}\eta$ as the Hamiltonian, making the Reeb vector field $\mathcal{R}$ the Hamiltonian vector field. The Hamiltonian contact structure in this case is said to be of Reeb type \cite{Boyer2011completely}. From now on, we will focus on this case and introduce the corresponding complete integrability conditions.

Similar to the Poisson bracket in symplectic geometry, the Lie algebra structure of $C^\infty(M)$ in contact Hamiltonian setting is given by the Jacobi bracket
\begin{align}\label{Jacobi bracket}
[f,g]_\eta=-\iota_{X_f}\iota_{X_g}d\eta+f\iota_{\mathcal{R}}dg-g\iota_{\mathcal{R}}df.
\end{align}
It is worth mentioning that this bracket is bilinear, antisymmetric, and satisfies the Jacobi identity. Furthermore, it fulfils a ``weak" Leibniz rule:
\begin{align}%\label{Jacobi bracket}
[f,gh]_\eta=[f,g]_\eta h+g[f,h]_\eta-[f,1]_\eta,
\end{align}
which is the key difference with the Poisson bracket.

If a function $h$ satisfies that $[h,1]_\eta=-\iota_{\mathcal{R}}dh=0$, it represents a first integral of the vector field $\mathcal{R}$. A contact Hamiltonian structure of Reeb type is said to be completely integrable if there exists $(n+1)$ first integrals $h_0=1,h_1,\cdots,h_n$ that are independent and in involution. That is,
the corresponding  Hamiltonian vector fields are linearly independent at almost all points and
\begin{align}\label{Integrability}
[h_i,1]_\eta=[h_i,h_j]_\eta=0,\;\;i,j=1,2,\cdots,n.
\end{align}
%\begin{remark}
We remark that definition of complete integrability can be generalized to any Hamiltonian. See Boyer \cite{Boyer2011completely} for details. We point out that, for the contact form $\eta$ and a given Hamiltonian $H$,  a function $f$ commute with $H$ may not be equivalent to $f$ being a first integral, unless $H$ is constant along the flow of the Reeb vector field.

We also note that a completely integrable contact Hamiltonian system is said to be of toric type, if the corresponding vector fields $X_{h_0}=\mathcal{R},X_{h_1},\cdots,X_{h_n}$ form the Lie algebra of a torus $\TT^{n+1}$. Analogous to that of the celebrated Arnold-Liouville theorem in the symplectic setting, it is possible to introduce the generalized action-angle variables in contact manifolds. The action of a torus $\TT^{n+1}$ on a $(2n+1)$-dimensional contact manifold $(M,\eta)$ is completely integrable if it is effective and preserve the contact structure \cite{Lerman2001}.\\
%\end{remark}

We call \eqref{Equation-1} a stochastic completely integrable contact Hamiltonian system, if the family of $d$ (for convenience, we set $d=n+1$ here) smooth contact Hamiltonians $\{ H_k \}_{k=1}^{d}$ form a completely integrable system in the sense that they are independent and in involution, and if $H_0$ commutes with this family under Jacobi bracket.

%%------------------------------------------------------------------------------------------------------------------------------------------
\begin{example}
{\bf(A stochastic dissipative mechanical system)}  We consider the product manifold $M=\RR\times T^\star\RR^2$ endowed with the canonical contact structure $\eta=z-p_1dq_1-p_2dq_2$, where $z$ and $(q,p) = (q_1,q_2,p_1,p_2)$ are global coordinates on $\RR$ and $T^\star\RR^2$, respectively. Let $H:M\to\RR$ be a contact Hamiltonian function given by
\begin{align}
H(q,p,z)=\frac{1}{2m}(p_1^2+p_2^2)+V(q)+\gamma z, % \;\; \text{with}\;\;\gamma\in\RR-\{0\}.
\end{align}
with $V(q)$ a potential function and $\gamma$ a positive constant. This Hamiltonian corresponds to a system with a friction force that depends linearly on the velocity (in this case, on the momentum). When such a contact Hamiltonian system is perturbed by noise, we have the following equation:
\begin{align}\label{SDMS1}
\frac{dq_i}{dt}=\frac{p_i}{m},\;\frac{dp_i}{dt}=-\frac{\partial V}{\partial q_i}-\gamma p_i,\;i=1,2,\;\text{and}\;%\;\;i=1,2, \notag\\
\frac{dz}{dt}=\frac{1}{m}(p_1^2+p_2^2)-H+\varepsilon \dot{B}_t
\end{align}
or, equivalently,
\begin{align}\label{SDMS2}
m\ddot{q}_i+\gamma m\dot{q}_i+\frac{\partial V}{\partial q_i}=0,\;i=1,2,\;\text{and}\;
\frac{dz}{dt}=\frac{1}{2}m(\dot{q}_1^2+\dot{q}_2^2)-V(q)-\gamma z+\varepsilon \dot{B}_t,
\end{align}
where $\dot{B}_t$ is a Gaussian white noise and $\varepsilon$ is the noise intensity.
With $H_0=H$ and $H_1=\varepsilon$, equation \eqref{SDMS1} or \eqref{SDMS2} is indeed a stochastic contact Hamiltonian system. The stochastic flow preserve the contact structure with conformal factor $\lambda_t=e^{-\gamma(t-t_0)}$. Furthermore, the functions $H_0$ and $H_1$ are in involution, i.e., $[H_0,H_1]_\eta=0$.

\end{example}

%%------------------------------------------------------------------------------------------------------------------------------------------
\begin{example}
{\bf(A stochastic contact Hamiltonian system in Sasaki-Einstein space)}  The homogeneous toric Sasaki-Einstein Space  $T^{1,1}$ was considered as the first example of toric Sasaki-Einstein/quiver duality and provides supersymmetric backgrounds relevant for the Ads/CFT correspondence \cite{Visinescu2017contact}. The space $T^{1,1}$ can be regarded as a $U(1)$ bundle over $S^2\times S^2$. We denote by $(\theta_i,\phi_i)$, $i=1,2,$ the coordinates which parametrize the two spheres $S^2$ in the conventional way, while the angle $\psi\in [0,4\pi)$ parametrizes the $U(1)$ fiber.

The globally defined contact 1-form $\eta_{SE}$ can be written in the form
\begin{align}\label{SE contact}
\eta_{SE}=\frac{1}{3}\left( d\psi+cos\theta_1d\varphi_1+cos\theta_2d\varphi_2\right).
\end{align}
The Reeb vector field defined by $\eta_{SE}$ is
\begin{align}\label{SE R}
\mathcal{R}_{SE}=3\frac{\partial}{\partial\psi}.
\end{align}
The contact Hamiltonian vector field $X_{H_k}$ with Hamiltonian $H_k$ can be expressed as
\begin{align}\label{SE vf}
X_{H_k}=&3\sum_{i}\frac{1}{\sin\theta_i}\left(\frac{\partial {H_k}}{\partial\varphi_i} -\frac{\partial {H_k}}{\partial\psi}\cos\theta_i\right)\frac{\partial}{\partial\theta_i}
-3\sum_{i}\frac{1}{\sin\theta_i}\frac{\partial {H_k}}{\partial\theta_i}\frac{\partial}{\partial\varphi_i}\notag\\
&+3\left({H_k}+\sum_{i}\frac{\cos\theta_i}{\sin\theta_i}\frac{\partial {H_k}}{\partial\theta_i}\right)\frac{\partial}{\partial\psi}.
\end{align}

Consider the following stochastic contact Hamiltonian system
\begin{eqnarray}\label{SE SHS}
d\begin{pmatrix}\theta_1 \\ \theta_2 \\ \varphi_1\\\varphi_2\\\psi\end{pmatrix}
=\begin{pmatrix}0 \\ 0 \\ 0\\ 0\\3\end{pmatrix}dt
+\begin{pmatrix}0 &0&0&\frac{3}{\sin\theta_1}&0\\ 0&0&0&0&\frac{3}{\sin\theta_2} \\ 0&1&0&0&0\\ 0&0&1&0&0\\3&0&0&3\varphi_1&3\varphi_2\end{pmatrix}\circ dB_t,
\end{eqnarray}
where $B_t$ is a 5-dimensional Brownian motion.
We find that the corresponding Hamiltonians $H_0=H_1=1$, $H_2=\frac{1}{3}\cos\theta_1$, $H_3=\frac{1}{3}\cos\theta_1$, $H_4=\varphi_1$ and $H_5=\varphi_2$, are independent and in involution. In particular, the constant Hamiltonian 1 corresponds to the Reeb vector field $\mathcal{R}_{SE}$ in \eqref{SE R}. Therefore, equation \eqref{SE SHS} is a stochastic completely integrable contact Hamiltonian system and the stochastic flow preserves the contact structure.

Moreover, using the analysis from Visinescu \cite{Visinescu2017contact}, we can choose $\mathcal{T}$ as a compact connected component of the level set $\{H_1=1, H_2 =c_1, H_3 =c_2\},$ where $c_1,c_2$ are positive constants. Then $\mathcal{T}$ is diffeomorphic to a $\TT^3$ torus. Let $D$ be an open domain in $\RR^2$. There exist a neighborhood $U$ of $\mathcal{T}$ and a diffeomorphism $\Phi:\;U\to\mathcal{T}\times D$
\begin{align}\label{SE AA}
\Phi(x)=(\vartheta_1,\vartheta_2,\vartheta_3,y_1,y_2).
\end{align}
More precisely, we can introduce the angle variables
\begin{align}\label{SE Angle}
\vartheta_0=\frac{\psi}{3},\;\vartheta_1=\varphi_1,\;\vartheta_2=\varphi_2,
\end{align}
and the generalized action variables
\begin{align}\label{SE Action}
y_0=1,\;y_1=\frac{1}{3}\cos\theta_1,\;y_2=\frac{1}{3}\cos\theta_2.
\end{align}
Therefore, equation \eqref{SE SHS} becomes simplier in these coordinates:
\begin{eqnarray}\label{SE AA}
d\begin{pmatrix}y_1 \\ y_2 \\ \vartheta_1\\ \vartheta_2\\ \vartheta_0\end{pmatrix}
=\begin{pmatrix}0 \\ 0 \\ 0\\ 0\\1\end{pmatrix}dt
+\begin{pmatrix}0&0&0&1&0\\ 0&0&0&0&1 \\ 0&1&0&0&0\\ 0&0&1&0&0\\1&0&0&\vartheta_1&\vartheta_2\end{pmatrix}\circ dB_t.
\end{eqnarray}
Note that the contact form has the canonical expression \cite{Visinescu2017contact} $\eta_0=(\Phi^{-1})^\ast\eta_{SE}=\sum_{i=0}^3y_id\vartheta_i.$
We call the local coordinates $(\vartheta_i,y_i)$ the generalized contact action-angle coordinates.

%where $\{B_t^k\}_{k=1}^5$ are pairwise independent Brownian motions.
\end{example}

%\section*{Appendix}

\section*{DATA AVAILABILITY}
The data that support the findings of this study are available within the article.

\section*{Acknowledgments}
The authors would like to thank Prof Maosong Xiang, Dr Lingyu Feng, Dr Jianyu Hu, Dr Li Lv and Dr Jun Zhang for helpful discussions. This work was partly supported by NSFC grants 11771449 and 11531006.

% If in two-column mode, this environment will change to single-column format so that long equations can be displayed.
%\begin{widetext}
%$$\mbox{put long equation here}$$
%\end{widetext}

% Figures should be put into the text as floats.
% Use the graphics or graphicx packages (distributed with LaTeX2e).
% Put the label that you will use with \ref{} command in the braces of the \label{} command.
%
% \begin{figure}
% \end{figure}

% \begin{table}
% \end{tabular}
% \end{table}

%\begin{acknowledgments}
% Put your acknowledgments here.
%\end{acknowledgments}

\bibliographystyle{unsrt}
\bibliography{ref-yuan}

\begin{thebibliography}{10}

\bibitem{Bravetti2017entropy}
Alessandro Bravetti.
\newblock Contact {H}amiltonian dynamics: The concept and its use.
\newblock {\em Entropy}, 19(10):535, 2017.

\bibitem{Arnold2013mathematical}
Vladimir~Igorevich Arnol'd.
\newblock {\em Mathematical methods of classical mechanics}, volume~60.
\newblock Springer Science \& Business Media, 2013.

\bibitem{DeLeon2019}
Manuel de~Le{\'o}n and Manuel Lainz~Valc{\'a}zar.
\newblock Contact {H}amiltonian systems.
\newblock {\em Journal of Mathematical Physics}, 60(10):102902, 2019.

\bibitem{Bravetti2019contact}
Alessandro Bravetti.
\newblock Contact geometry and thermodynamics.
\newblock {\em International Journal of Geometric Methods in Modern Physics},
  16:1940003, 2019.

\bibitem{Ramirez2016partial}
Hector Ramirez, Bernhard Maschke, and Daniel Sbarbaro.
\newblock Partial stabilization of input-output contact systems on a legendre
  submanifold.
\newblock {\em IEEE Transactions on Automatic Control}, 62(3):1431--1437, 2017.

\bibitem{Ghrist2007contact}
Robert Ghrist.
\newblock On the contact topology and geometry of ideal fluids.
\newblock In {\em Handbook of Mathematical Fluid Dynamics}, volume~4, pages
  1--37. Elsevier, 2007.

\bibitem{Wang2019aubry}
Kaizhi Wang, Lin Wang, and Jun Yan.
\newblock Aubry--{M}ather theory for contact {H}amiltonian systems.
\newblock {\em Communications in Mathematical Physics}, 366(3):981--1023, 2019.

\bibitem{Boyer2011completely}
Charles~P Boyer.
\newblock Completely integrable contact hamiltonian systems and toric contact
  structures on ${S}^2 \times {S}^3$.
\newblock {\em Symmetry, Integrability and Geometry: Methods and Applications},
  7:058, 2011.

\bibitem{Visinescu2017contact}
Mihai Visinescu.
\newblock Contact {H}amiltonian systems and complete integrability.
\newblock In {\em AIP Conference Proceedings}, volume 1916, page 020002. AIP
  Publishing LLC, 2017.

\bibitem{Wang2019variational}
Ya-Nan Wang and Jun Yan.
\newblock A variational principle for contact {H}amiltonian systems.
\newblock {\em Journal of Differential Equations}, 267(7):4047--4088, 2019.

\bibitem{Liu2018contact}
Qihuai Liu, Pedro~J Torres, and Chao Wang.
\newblock Contact {H}amiltonian dynamics: Variational principles, invariants,
  completeness and periodic behavior.
\newblock {\em Annals of Physics}, 395:26--44, 2018.

\bibitem{Vermeeren2019contact}
Mats Vermeeren, Alessandro Bravetti, and Marcello Seri.
\newblock Contact variational integrators.
\newblock {\em Journal of Physics A: Mathematical and Theoretical},
  52(44):445206, 2019.

\bibitem{Duan}
Jinqiao Duan.
\newblock {\em An Introduction to Stochastic Dynamics}.
\newblock Cambridge University Press, 2015.

\bibitem{Brin2000stochastic}
Michael Brin and Mark Freidlin.
\newblock On stochastic behavior of perturbed {H}amiltonian systems.
\newblock {\em Ergodic Theory and Dynamical Systems}, 20(1):55--76, 2000.

\bibitem{mackay2010langevin}
RS~MacKay.
\newblock Langevin equation for slow degrees of freedom of hamiltonian systems.
\newblock In {\em Nonlinear dynamics and chaos: advances and perspectives},
  pages 89--102. Springer, 2010.

\bibitem{Bi}
Jean-Michel Bismut.
\newblock {\em M{\'e}canique Al{\'e}atoire, Lecture Notes in Mathematics},
  volume 866.
\newblock Springer-Verlag, 1981.

\bibitem{Li}
Xue-Mei Li.
\newblock An {A}veraging principle for a completely integrable stochastic
  {H}amiltonian system.
\newblock {\em Nonlinearity}, 21(4):803, 2008.

\bibitem{Pingyuan2019}
Pingyuan Wei, Ying Chao, and Jinqiao Duan.
\newblock Hamiltonian systems with {L}\'evy noise: {S}ymplecticity,
  {H}amilton's principle and averaging principle.
\newblock {\em Physica D: Nonlinear Phenomena}, 398:69--83, 2019.

\bibitem{wang2009dynamics}
Lijin Wang, Jialin Hong, Rudolf Scherer, and Fengshan Bai.
\newblock Dynamics and variational integrators of stochastic hamiltonian
  systems.
\newblock {\em International Journal of Numerical Analysis \& Modeling}, 6(4),
  2009.

\bibitem{Bravetti2017contact}
Alessandro Bravetti, Hans Cruz, and Diego Tapias.
\newblock Contact {H}amiltonian mechanics.
\newblock {\em Annals of Physics}, 376:17--39, 2017.

\bibitem{Geiges2008}
Hansj{\"o}rg Geiges.
\newblock {\em An Introduction to Contact Topology}, volume 109.
\newblock Cambridge University Press, 2008.

\bibitem{Lerman2001}
Eugene Lerman.
\newblock Contact toric manifolds.
\newblock {\em Journal of Symplectic Geometry}, 1(6):785--828, 2001.

\end{thebibliography}

\end{document}